\documentclass[11pt]{article}

\usepackage{amsfonts,mathrsfs,amsthm,amsmath,graphicx}

\usepackage{enumerate}
\usepackage{bookmark}

\newtheorem{theorem}{Theorem}
\newtheorem{cor}[theorem]{Corollary}

\newtheorem{ex}[theorem]{Example}
\newtheorem{lemma}[theorem]{Lemma}
\newtheorem{definition}[theorem]{Definition}

\newcommand{\fix}{{\rm fix\,}}
\newcommand{\cyc}{{\rm cyc\,}}
\newcommand{\red}{{\rm red\,}}

\newcommand{\exc}{{\rm exc\,}}
\newcommand{\arcc}{{\rm arc\,}}
\newcommand{\com}{{\rm com\,}}
\newcommand{\down}{{\rm down\,}}
\newcommand{\ver}{{\rm ver\,}}

\newcommand{\cdes}{{\rm cdes\,}}

\newcommand{\nega}{{\rm neg\,}}

\newcommand{\msn}{\mathfrak{S}_n}

\begin{document}
\title{The cycle descent statistic on permutations}
\author{
Jun Ma$^{a,}$\thanks{Email address: majun904@sjtu.edu.cn
\newline{\hspace*{1.8em}Partially supported by  SRFDP 20110073120068 and NSFC 11571235}}\\
\and Shimei Ma$^{b,}$\thanks{Email address of the corresponding
author: shimeimapapers@163.com
\newline{\hspace*{1.8em}Partially supported by  NSFC 11401083}}\\
 \and Yeong-Nan Yeh$^{c,}$\thanks{Email address: mayeh@math.sinica.edu.tw\newline{\hspace*{1.8em}Partially supported by NSC 104-2115-M-001-010-MY3}}\\
\and Zhu Xu$^{a,}$\thanks{Email address:
s\_j\_z\_x@sjtu.edu.cn}}\date{} \maketitle
\vspace*{-1.2cm}\begin{center} \footnotesize $^{a}$ Department of
Mathematics, Shanghai Jiaotong University,
Shanghai\\
$^{b}$School of mathematics and statistics,
        northeastern university at qinhuangdao,
         Hebei\\
$^{c}$ Institute of Mathematics, Academia Sinica, Taipei\\

\end{center}

 \vspace*{-0.3cm}
\thispagestyle{empty}
\begin{abstract}
In this paper we study the cycle descent statistic on permutations.
Several involutions on permutations and derangements are
constructed. Moreover, we construct a bijection between negative
cycle descent permutations and Callan perfect matchings.
\end{abstract}

\noindent {\bf MSC:} 05A15; 05A19\\
{\bf Keywords:} Permutations; Cycle descents; Perfect matchings;
Eulerian polynomials

\section{Introduction}
Let $\mathfrak{S}_n$ be the symmetric group of all permutations of
$[n]$, where $[n]=\{1,2,\ldots,n\}$. We write an element $\pi$ in
$\mathfrak{S}_n$ as $\pi=\pi(1)\pi(2)\cdots\pi(n)$. An {\it
excedance} in $\pi$ is an index $i$ such that $\pi(i)>i$ and  a {\it
fixed point} in $\pi$ is an index $i$ such that $\pi(i)=i$. A
fixed-point-free permutation is called a {\it derangement}. Denote
by $\mathscr{D}_n$ the set of derangements of $[n]$. As usual, let
$\exc(\pi)$, $\fix(\pi)$ and $\cyc(\pi)$ denote the number of
excedances, fixed points and cycles in $\pi$ respectively. For
example, the permutation $\pi=3142765$ has the cycle decomposition
$(1342)(57)(6)$, so $\cyc(\pi)=3$, $\exc(\pi)=3$ and $\fix(\pi)=1$.

The Eulerian polynomials $A_n(x)$ are defined by
$$A_0(x)=1,\quad A_n(x)=\sum_{\pi\in\msn}x^{\exc(\pi)}\quad\textrm{for $n\ge 1$},$$
and have been extensively investigated.  Foata and Sch\"utzenberger
\cite{FS70} introduced a $q$-analog of the Eulerian polynomials
defined by
\begin{equation*}\label{anxq-def}
A_n(x;q)=\sum_{\pi\in\msn}x^{\exc(\pi)}q^{\cyc(\pi)}.
\end{equation*}
Brenti~\cite{Bre94,Bre00} further studied $q$-Eulerian polynomials
and established the link with $q$-symmetric functions arising from
plethysm. Brenti~\cite[Proposition 7.3]{Bre00} obtained the
exponential generating function for $A_n(x;q)$:
\begin{equation*}\label{anxq-rr}
1+\sum_{n\geq
1}A_n(x;q)\frac{z^n}{n!}=\left(\frac{1-x}{e^{z(x-1)}-x}\right)^q.
\end{equation*}
Remarkably, Brenti~\cite[Corollary 7.4]{Bre00} derived the following
identity:
\begin{equation}\label{Brenti}
\sum_{\pi\in\msn}x^{\exc(\pi)}(-1)^{\cyc(\pi)}=-(x-1)^{n-1}.
\end{equation}
From then on, there is a large of literature devoted to various
generalizations and refinements of the joint distribution of
excedances and cycles (see~\cite{Bagno07,Chen09,Ksavrelof03,Zhao13}
for instance). For example, Ksavrelof and Zeng~\cite{Ksavrelof03}
constructed bijective proofs of~\ref{Brenti} and the following
formula:
\begin{equation*}
\sum_{\pi\in
\mathscr{D}_n}x^{\exc(\pi)}(-1)^{\cyc(\pi)}=-x-x^2-\cdots-x^{n-1}.
\end{equation*} In
particular, their bijection leads to a refinement of the above
identity:
\begin{equation*}
\sum_{\pi\in\mathscr{D}_{n,i}}x^{\exc(\pi)}(-1)^{\cyc(\pi)}=-x^{n-i},
\end{equation*}
where $\mathscr{D}_{n,i}$ is the set of derangements $\pi$ of $[n]$
such that $\pi(n)=i$.

A {\it standard cycle decomposition} of $\pi\in\mathfrak{S}_n$ is
defined by requiring that each cycle is written with its smallest
element first, and the cycles are written in increasing order of
their smallest element. A permutation is said to be cyclic if there
is only one cycle in its cycle decomposition. Let
$(c_1,c_2,\ldots,c_i)$ be a cycle in standard cycle decomposition of
$\pi$. We say that $c_j$ is a {\it cycle descent} if $c_j>c_{j+1}$,
where $1<j<i$. Denote by $CDES(\pi)$ the set of cycle descents of
$\pi$ and let $\cdes(\pi)=|CDES(\pi)|$ be the number of cycle
descents of $\pi$. For example, for $\pi=(1342)(57)(6)$, we have
$CDES(\pi)=\{4\}$ and $\cdes(\pi)=1$. For $\pi\in\msn$, it is clear
that $\exc(\pi)+\cyc(\pi)+\cdes(\pi)=n$. Thus
$$A_n(x;q)=q^n\sum_{\pi\in\msn}\left(\frac{x}{q}\right)^{\exc(\pi)}\left(\frac{1}{q}\right)^{\cdes(\pi)}.$$

Let $\mathfrak{S}_{n,i}$ be the set of permutations
$\pi\in\mathfrak{S}_n$ with $\pi(i)=1$. For any
$\pi\in\mathfrak{S}_{n}$, let $\pi^{-1}$ denote the inverse of
$\pi$, so $\pi^{-1}(1)=i$ if $\pi\in\mathfrak{S}_{n,i}$. For $n\geq
2$, we recently observed the following formulas:
\begin{align*}
\sum_{\pi\in \mathfrak{S}_{n,i}}(-1)^{\cdes(\pi)}t^{\pi^{-1}(1)}&=
\left\{\begin{array}{lll}2^{n-2}t&\text{if}&i=1,\\0&\text{if}&i=2,\ldots,n-1,\\2^{n-2}t^{n}&\text{if}&i=n,\end{array}\right.\\
\sum_{\pi\in
\mathfrak{S}_{n}}(-1)^{\cdes(\pi)}t^{\pi^{-1}(1)}&=2^{n-2}(t+t^{n}),\\
\sum\limits_{\pi\in \mathscr{D}_{n}}x^{exc(\pi)}(-1)^{cdes(\pi)}t^{\pi^{-1}(1)}&=\sum_{i=2}^{n}(-1)^{n-i}x^{i-1}t^i,\\
\sum_{\pi\in
\mathscr{D}_{n}}(-1)^{cdes(\pi)}&=\frac{1}{2}[1-(-1)^{n-1}].
\end{align*}
The above formulas can be easily proved by taking $x=1$ in
Theorem \ref{theorem-p} of Section $2$. Motivated by the these
formulas, we shall study the cycle descent statistic of
permutations. In the next section, we present the main results of
this paper and collect some notation and definitions that will be
needed in the rest of the paper.

\section{Definitions and main results}\label{section02}
Consider the following enumerative polynomials
$$P_{n,i}(x,y,q,t)=\sum\limits_{\pi\in
\mathfrak{S}_{n,i}}x^{\exc(\pi)}y^{\cdes(\pi)}q^{\fix(\pi)}t^{\pi^{-1}(1)}.$$
It is remarkable that the polynomials $P_{n,i}(x,-1,1,t)$ and
$P_{n,i}(x,-1,0,t)$ have simple closed formulas. We state them as
the first main result of this paper.
\begin{theorem}\label{theorem-p}
For $n\geq 2$, we have
\begin{equation}\label{P_{n,i}(x,-1,1,t)}
P_{n,i}(x,-1,1,t)=\left\{\begin{array}{lll}t(1+x)^{n-2}&\text{if}&i=1,\\0&\text{if}&i=2,\ldots,n-1,\\t^{n}x(1+x)^{n-2}&\text{if}&i=n,\end{array}\right.
\end{equation} and
\begin{equation}\label{P_{n,i}(x,-1,0,t)}P_{n,i}(x,-1,0,t)=(-1)^{n-i}x^{i-1}t^i.
\end{equation}
\end{theorem}

A signed permutation $(\pi, \phi)$ of $[n]$ is a permutation
$\pi\in\msn$ together with a map $\phi : [n] \mapsto \{+1,-1\}$ and
we call $\phi(i)$ the sign of $i$. For simplicity, we indicate the
sign of $\pi(i)$ by writing $\pi(i)^+$ or $\pi(i)^-$. Let $(\pi,
\phi)$ be a signed permutation. Let $NEG(\pi, \phi)$ be the set of
numbers $\pi(i)$ with the sign $-1$, i.e.
$$NEG(\pi, \phi)=\{\pi(i)\mid \phi(\pi(i))=-1\},$$ and let $neg(\pi,\phi)=|NEG(\pi,
\phi)|$.

\begin{definition}
A negative cycle descent permutation $(\pi, \phi)$ of $[n]$ is a
signed permutation $(\pi, \phi)$ such that $NEG(\pi,\phi)\subseteq
CDES(\pi)$.
\end{definition}

Let
$$b_{n}(y,q)=\sum\limits_{i=1}^nP_{n,i}(1,y,q,1)=\sum\limits_{\pi\in
\mathfrak{S}_{n}}y^{\cdes(\pi)}q^{\fix(\pi)}.$$ It is easy to verify
that $b_n(2,1)$ is the number of negative cycle descent permutations
of $[n]$ since
$$b_n(2,1)=\sum\limits_{\pi\in \mathfrak{S}_{n}}2^{\cdes(\pi)}$$ and
$b_n(2,0)$ is the number of negative cycle descent derangements of
$[n]$ since
  $$b_n(2,0)=\sum\limits_{\pi\in
\mathfrak{S}_{n}}2^{\cdes(\pi)}0^{\fix(\pi)}=\sum\limits_{\pi\in
\mathscr{D}_{n}}2^{\cdes(\pi)}.$$

We present the second main result of this paper as follows.
\begin{theorem}\label{theorem-b}
For $n\geq 1$, we have
\begin{equation}\label{P_{n,i}(1,y,1,1)-recurrent}
b_{n+1}(y,1)=b_{n}(y,1)+\sum_{i=1}^nb_{i}(y,1)\binom{n}{i-1}(y-1)^{n-i}
\end{equation}
with the initial condition $b_1(y,1)=1$,  and
\begin{equation}\label{P_{n,i}(1,y,0,1)-recurrent}
b_{n+1}(y,0)=\sum\limits_{i=0}^{n-1}{\binom{n}{
i}}\left[b_{i+1}(y,0)+b_{i}(y,0)\right](y-1)^{n-i-1}
\end{equation}
with initial conditions  $b_0(y,0)=1, b_1(y,0)=0$.
\end{theorem}

By taking $y=2$ in the identity~\ref{P_{n,i}(1,y,1,1)-recurrent}, we
obtain Klazar's recurrence for $w_{12}(n)$
(see~\cite[Eq.~(39)]{Klazar97} for details), which can be written as
follows:\begin{equation}\label{Klazar-recurrence}b_{n+1}(2,1)=b_n(2,1)+\sum\limits_{i=1}^nb_{n+1-i}(2,1){\binom{n}{i}}.\end{equation}
In \cite{Callan10, Klazar97, Ren15}, the sets of some combinatorial
objects, which have cardinality $b_n(2,1)$, were studied. We list
some of them as follows:
\begin{itemize}
  \item [\rm (i)] The set of drawings of rooted plane trees with $n+1$ vertices (see~\cite{Klazar97});
  \item [\rm (ii)] The set of Klazar trees with $n+1$ vertices (see~\cite{Callan10});
  \item [\rm (iii)] The set of perfect matchings on the set $[2n]$ in which no even number is matched to a larger odd number (see~\cite{Callan10}).
  \item [\rm  (iv)] The set of ordered partitions of $[n]$ all of whose left-to-right minima occur at odd locations (see~\cite{Ren15}).
\end{itemize}

Now we begin to introduce the concept of perfect matchings. Let
$\mathbb{P}_A=A\times \{0,1\}$, where $A=\{i_1,\ldots,i_k\}$ is a
finite set of positive integers with $i_1<i_2<\cdots<i_k$. When
$A=[n]$, we write $\mathbb{P}_A$ as $\mathbb{P}_n$. A perfect
matching is a partition of $\mathbb{P}_A$ into $2$-element subsets
or matches.  For any match $\{(i,x),(j,y)\}$ in a perfect matching,
we say that $(i,x)$ is the {\it partner} of $(j,y)$. For
convenience, we represent a perfect matching as a dot diagram with
vertices arranged in two rows.
\begin{ex}\label{example-dot-graph}We give a dot diagram of a perfect matching $M$ of
$\mathbb{P}_8$ as follows:\begin{center}
\includegraphics[width=8cm,height=2.5cm]{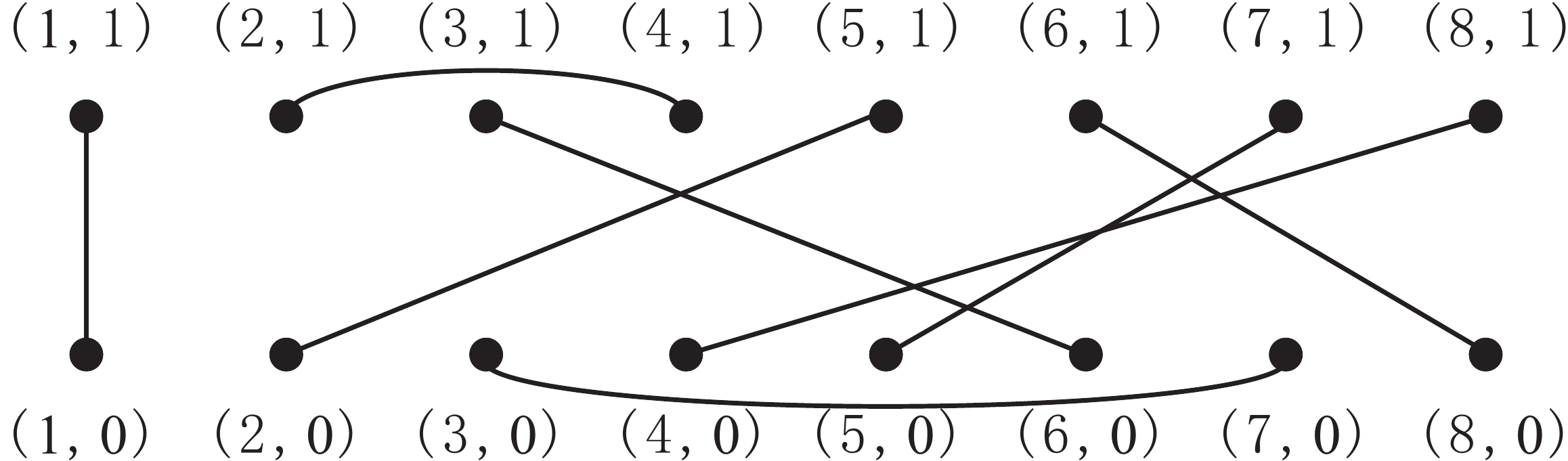}\\
Fig.1. A perfect matching $M$ of $\mathbb{P}_8$
\end{center}
\end{ex}
\noindent Thus, for any perfect matching $M$ of $\mathbb{P}_A$, we
say that $\mathbb{P}_A$ is the vertex set of $M$ and every match is
an edge of $M$. We use $V(M)$ and $E(M)$ to denote vertices set and
edges set in $M$ respectively. Moreover, an edge is called an {\it
arc} if it joins two dots in the same row; otherwise, this edge is
called a {\it line}. For any line $\{(i,0),(j,1)\}$, it is said to
be a {\it upline} if $i<j$, a {\it downline} if $i>j$ and a {\it
vertical} line if $i=j$. For any perfect matching $M$, let
$\arcc(M), \down(M)$ and $\ver(M)$ be the numbers of arc, down lines
and vertical lines in $M$ respectively.
\begin{ex} In the perfect matching of Example~\ref{example-dot-graph},
the edge $\{(1,1),(1,0)\}$ is a vertical line, the edges
$\{(3,1),(6,0)\}$ and $\{(6,1),(8,0)\}$ are two downlines, the edges
$\{(5,1),(2,0)\}$, $\{(7,1),(5,0)\}$ and $\{(8,1),(4,0)\}$ are three
uplines, the edges $\{(2,1),(4,1)\}$ and $\{(3,0),(7,0)\}$ are two
arcs; finally, $arc(M)=2,down(M)=2,ver(M)=1$.
\end{ex}
\begin{definition} A perfect matching $M$ of $\mathbb{P}_n$ is
a Callan perfect matching if $M$ with no uplines.
\end{definition}
\begin{ex}\label{example-callan-matching}
We give a dot diagram of a Callan perfect matching $M$ of
$\mathbb{P}_8$ as follows:\begin{center}
\includegraphics[width=8cm,height=2.5cm]{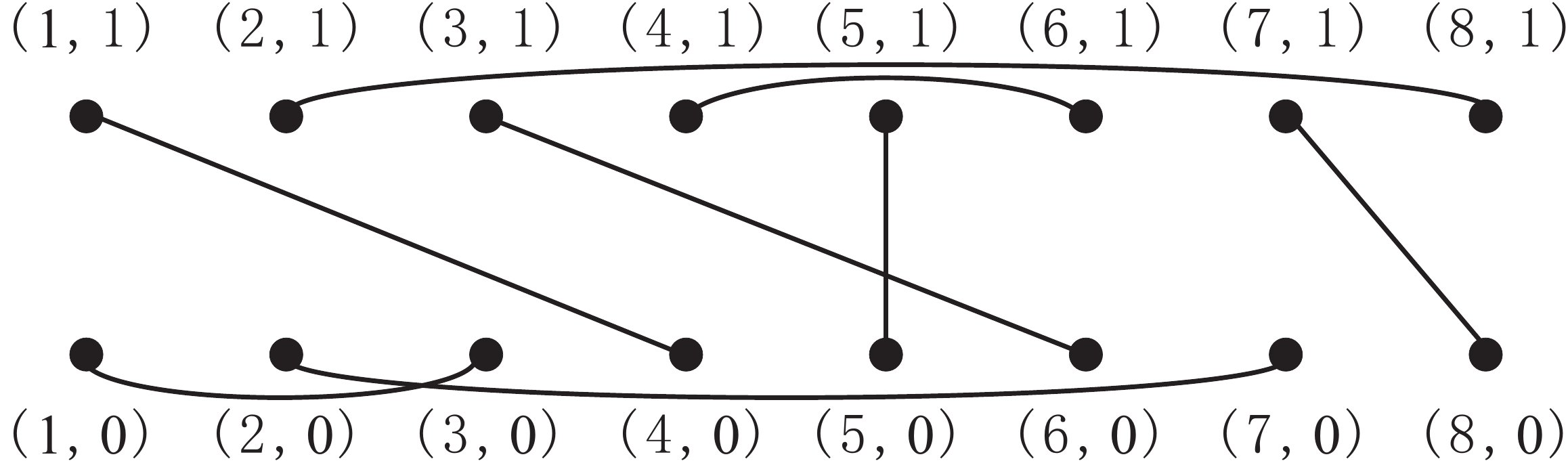}\\
Fig.2. A Callan perfect matching $M$ of $\mathbb{P}_8$
\end{center}
\end{ex}
\noindent Let $m_n$ be the number of Callan perfect matchings of
$\mathbb{P}_n$. Callan~\cite{Callan10} proved that $m_n$ satisfies
the recurrence (\ref{Klazar-recurrence}). So the number of negative
cycle descent permutations of $[n]$ equals to the number of Callan
perfect matching of $\mathbb{P}_n$.

Let $M$ be a perfect matching of $\mathbb{P}_n$. We say that $M'$ is
a {\it sub-perfect matching} of $M$ if $M'$ is a perfect matching
such that $V(M')\subseteq V(M)$ and $E(M')\subseteq E(M)$. For any
$V\subseteq [n]$, if there is a  sub-perfect matching $M'$ of $M$
with $V(M')=V\times \{0,1\}$, then $M'$ is said to be the
sub-perfect matching induced by $V$ and is denoted by $M[V]$.

Denote by $\mathcal {G}(M)$ a graph which is obtained from $M$ by
identifying two vertices $(i,0)$ and $(i,1)$ as a new vertex $i$ for
any $i\in[n]$. It is easy to see that the graph $\mathcal {G}(M)$ is
the union of some disjoint cycles. For a cycle $C$ in
$\mathcal{G}(M)$, suppose $C$ has the vertices set $V$. Note that
there is a sub-perfect matching of $M$ induced by $V$. We say that
$M[V]$ is a connected component of $M$. Let $\com(M)$ be the number
of connected components in a perfect matching $M$. If a perfect
matching $M$ has exactly one connected component, i.e., $\com(M)=1$,
then we say that $M$ is a connected perfect matching.

\begin{ex}For the perfect matching $M$ of Example
\ref{example-callan-matching}, we draw the graph $\mathcal {G}(M)$
as follows:\begin{center}
\includegraphics[width=5.5cm,height=2cm]{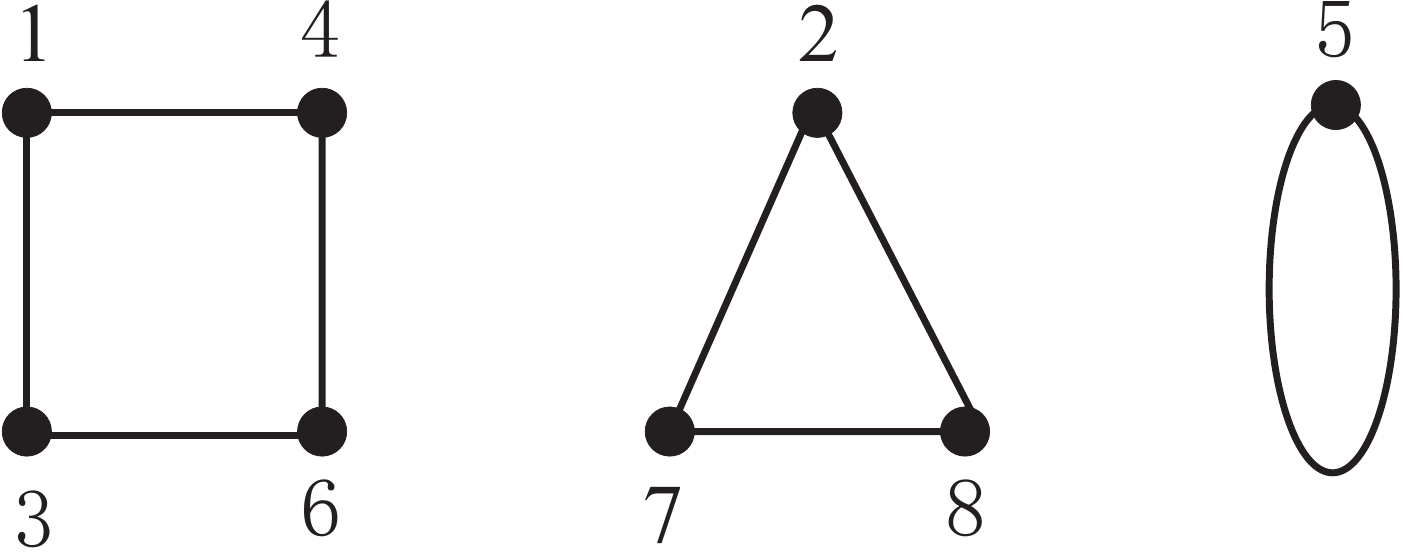}\\
Fig.3. A graph $\mathcal {G}(M)$.
\end{center}
So we have $\com(M)=3$.
\end{ex}

We  state the third main result of this paper as follows.
\begin{theorem}\label{theorem-bijection-p-m}
There is a bijection $\Gamma_n$ between the set of negative cycle
descent permutations of $[n]$ and the set of Callan perfect
matchings of $\mathbb{P}_n$. Moreover, for any negative cycle
descent permutation $(\pi,\phi)$ of $[n]$, we have
$$\com(\Gamma_n(\pi,\phi))=\cyc(\pi),~\ver(\Gamma_n(\pi,\phi))=\fix(\pi),$$ and
$$\down(\Gamma_n(\pi,\phi))=\left\{\begin{array}{ll}\nega(\pi,\phi)&\text{if
}(1,1)\text{ and its partner are in the same
row,}\\\nega(\pi,\phi)+1&\text{otherwise.}\end{array}\right.$$
\end{theorem}

Let $\Gamma|_{\mathscr{D}_n}$ denote the restriction of $\Gamma_n$
on the set of negative cycle descent derangements of $[n]$. So the
following corollary is immediate.

\begin{cor}\label{cor-bijection-d-p-m}
$\Gamma|_{\mathscr{D}_n}$ is a bijection between the set of negative
cycle descent derangements of $[n]$ and the set of Callan perfect
matchings of $\mathbb{P}_n$ which have no vertical lines.
\end{cor}

The rest of this paper is organized as follows. In
Section~\ref{section03} and Section~\ref{section04}, we
prove~\ref{P_{n,i}(x,-1,1,t)} and~\ref{P_{n,i}(x,-1,0,t)} in
Theorem~\ref{theorem-p} respectively. In Section~\ref{section05} and
Section~\ref{section06}, we prove~\ref{P_{n,i}(1,y,1,1)-recurrent}
and ~\ref{P_{n,i}(1,y,0,1)-recurrent} in Theorem~\ref{theorem-b}
respectively. In Section~\ref{section07}, we construct the bijection
$\Gamma_n$ in Theorem~\ref{theorem-bijection-p-m}.

\section{Proof of the explicit formula~\ref{P_{n,i}(x,-1,1,t)} in Theorem~\ref{theorem-p}}\label{section03}
Suppose that $\pi=\pi(i_1)\ldots\pi(i_k)$ is a permutation on the
set $\{i_1,\ldots,i_k\}$ of positive integers with
$i_1<i_2<\cdots<i_k$. Throughout this paper, we always let
$$\red(\pi):=\red(\pi(i_1))\ldots \red(\pi(i_k))\in\mathfrak{S}_k,$$
where $\red$ is a increasing map from $\{i_1,\ldots,i_k\}$ to
$\{1,2,\ldots,k\}$ defined by $\red(i_j)=j$ for any
$j=1,2,\ldots,k$.

Let $P_{n}(x,y,1,1)=\sum\limits_{i=1}^nP_{n,i}(x,y,1,1)$. We give a
recurrence for $P_{n,i}(x,y,1,1)$ in the following lemma.
\begin{lemma}\label{lemma-P(n,i)(x,y,1,1)-recurrent}For any $n\geq 2$, we
have $$P_{n+1,1}(x,y,1,1)=P_{n}(x,y,1,1)$$ and
$$P_{n+1,i}(x,y,1,1)=xP_{n-1}(x,y,1,1)+x\sum\limits_{j=2}^{i-1}P_{n,j}(x,y,1,1)+y\sum\limits_{j=i}^{n}P_{n,j}(x,y,1,1)$$
for $i=2,\ldots,n+1$  with initial conditions
$$P_{1,1}(x,y,1,1)=1,P_{2,1}(x,y,1,1)=1,P_{2,2}(x,y,1,1)=x.$$
\end{lemma}
\begin{proof}
For any $\pi=\pi(1)\pi(2)\ldots\pi(n+1)\in\mathfrak{S}_{n+1,1}$, we
have $\pi(1)=1$. Let $\tilde{\pi}=\pi(2)\ldots\pi(n+1)$. Then
$\tilde{\pi}$ is a permutation on the set $\{2,3,\ldots,n\}$ and
$\red(\tilde{\pi})\in\mathfrak{S}_n$. Obviously,
$$\exc(\pi)=\exc(\red(\tilde{\pi}))\text{ and }\cdes(\pi)=\cdes(\red(\tilde{\pi})).$$
So we have $P_{n+1,1}(x,y,1,1)=P_n(x,y,1,1)$.

For any $i\geq 2$, let
$\pi=\pi(1)\pi(2)\ldots\pi(n+1)\in\mathfrak{S}_{n+1,i}$. Let
$\sigma=(1,c_1,c_2,\ldots,c_m)$ be the cycle in the standard cycle
decomposition of $\pi$ which contains the entry $1$. So $\pi$ can be
split into the cycle $\sigma$ and a permutation $\tau$ on the set
$\{1,2,\ldots,n+1\}\setminus\{1,c_1,\ldots,c_m\}$, i.e.,
$\pi=\sigma\cdot\tau$. Clearly, $m\geq 1$, $i\geq 2$ and $c_m=i$
since $\pi\in\mathfrak{S}_{n+1,i}$. We distinguish the following two
cases:

{\bf Case 1.} $m=1$.

Deleting the cycle $(1,c_1)=(1,i)$ from the standard cycle
decomposition of $\pi$, we obtain the permutation
$$\tau=\pi(2)\ldots\pi(i-1)\pi(i+1)\ldots\pi(n+1)$$ which is defined on the
set $\{1,2,\ldots,n+1\}\setminus\{1,i\}$. Note that
$\red(\tau)\in\mathfrak{S}_{n-1}$ and
$$\exc(\pi)=\exc(\red(\tau))+1,\cdes(\pi)=\cdes(\red(\tau)).$$
This provides the term $$xP_{n-1}(x,y,1,1).$$

{\bf Case 2.} $m\geq 2$.

Suppose that $c_{m-1}=j$ for some $2\leq j\leq n+1$. Deleting the
number $c_m=i$ from  the standard cycle decomposition of $\pi$, we
obtain a permutation
$$\tilde{\pi}=(1,c_1,\ldots,c_{m-1})\cdot\tau$$ which is defined on
the set $\{1,\ldots,i-1,i+1,\ldots,n+1\}$. Hence
$\red(\tilde{\pi})\in\mathfrak{S}_{n}$. Moreover, if $c_{m-1}=j\leq
i-1$, then
$$\red(\tilde{\pi})\in\mathfrak{S}_{n,j},
\exc(\pi)=\exc(\red(\tilde{\pi}))+1,\cdes(\pi)=\cdes(\red(\tilde{\pi})).$$
This provides the term $$x\sum\limits_{j=2}^{i-1}P_{n,j}(x,y,1,1).$$
If $c_{m-1}=j\geq i+1$, then
$$\red(\tilde{\pi})\in\mathfrak{S}_{n,j-1},
\exc(\pi)=\exc(\red(\tilde{\pi})),
\cdes(\pi)=\cdes(\red(\tilde{\pi}))+1.$$ This provides the term
$$y\sum\limits_{j=i}^{n}P_{n,j}(x,y,1,1).$$
In conclusion, for any $i\geq 2$ we have $$P_{n+1,i}(x,y,1,1)=
xP_{n-1}(x,y,1,1)+x\sum\limits_{j=2}^{i-1}P_{n,j}(x,y,1,1)+y\sum\limits_{j=i}^{n}P_{n,j}(x,y,1,1).$$
\end{proof}

\noindent{\bf A proof of the identity~\ref{P_{n,i}(x,-1,1,t)} in Theorem \ref{theorem-p}:}\\
 Note that
$$\sum_{\pi\in
\mathfrak{S}_{n,i}}x^{\exc(\pi)}(-1)^{\cdes(\pi)}t^{\pi^{-1}(1)}=t^{i}\sum_{\pi\in
\mathfrak{S}_{n,i}}x^{\exc(\pi)}(-1)^{\cdes(\pi)}=t^iP_{n,i}(x,-1,1,1).$$
In order to prove the identity (\ref{P_{n,i}(x,-1,1,t)}) in Theorem
\ref{theorem-p},  it is sufficient to show that
\begin{equation}\label{equation-b(n,i)(x,-1)-formula-proof}
P_{n,i}(x,-1,1,1)=\left\{\begin{array}{lll}(1+x)^{n-2}&\text{if}&i=1,\\0&\text{if}&i=2,\ldots,n-1,\\x(1+x)^{n-2},&\text{if}&i=n.\end{array}\right.
\end{equation}

\noindent{\bf (i) An induction proof of the explicit
formula~\ref{equation-b(n,i)(x,-1)-formula-proof}.}
\begin{proof} It is easy to verify that
$$P_{2,1}(x,-1,1,1)=1,P_{2,2}(x,-1,1,1)=x.$$ Assume that the explicit formula~\ref{equation-b(n,i)(x,-1)-formula-proof} holds
for any $2\leq k\leq n$. By
Lemma~\ref{lemma-P(n,i)(x,y,1,1)-recurrent}, we have
\begin{eqnarray*}P_{n+1,1}(x,-1,1,1)&=&P_n(x,-1,1,1)\\
&=&P_{n,1}(x,-1,1,1)+P_{n,n}(x,-1,1,1)\\
&=&(1+x)^{n-2}+x(1+x)^{n-2}=(1+x)^{n-1},\end{eqnarray*}

\begin{eqnarray*}&&P_{n+1,n+1}(x,-1,1,1)\\&=&xP_{n-1}(x,-1,1,1)+x\sum\limits_{j=2}^{n}P_{n,j}(x,-1,1,1)-\sum\limits_{j=n+1}^{n}P_{n,j}(x,-1,1,1)\\
&=&xP_{n-1,1}(x,-1,,1,1)+xP_{n-1,n-1}(x,-1,1,1)+xP_{n,n}(x,-1,1,1)\\
&=&x(1+x)^{n-2}+x^2(1+x)^{n-2}=x(1+x)^{n-1}\end{eqnarray*} and
\begin{eqnarray*}&&P_{n+1,i}(x,-1,1,1)\\&=&xP_{n-1}(x,-1,1,1)+x\sum\limits_{j=2}^{i-1}P_{n,j}(x,-1,1,1)-\sum\limits_{j=i}^{n}P_{n,j}(x,-1,1,1)\\
&=&xP_{n-1,1}(x,-1,1,1)+xP_{n-1,n-1}(x,-1,1,1)-P_{n,n}(x,-1,1,1)\\
&=&x(1+x)^{n-2}-P_{n,n}(x,-1,1,1)=0\end{eqnarray*} for any $2\leq
i\leq n$.

\noindent{\bf (ii) A bijection proof of the explicit formula~\ref{equation-b(n,i)(x,-1)-formula-proof}.}\\
Now we give a bijective proof
of~\ref{equation-b(n,i)(x,-1)-formula-proof} by establishing an
involution $\psi_{n,i}$ on $\mathfrak{S}_{n,i}$.

For any $\pi\in\mathfrak{S}_n$, suppose that $\pi=C_1\ldots C_k$ is
the standard cycle decomposition of $\pi$. Let
$$\hat{\pi}=a_1a_2\ldots a_n$$ be the permutation obtained from $\pi$
by erasing the parentheses in its standard cycle decomposition. If
$\hat{\pi}\neq 12\ldots n$, then $\hat{\pi}$ has some {\it
top-descent}, i.e., the number $a_i$ satisfying $a_i>a_{i+1}$, and
let $q_{\pi}$ be the {\it last top-descent} which appears in the
sequence $\hat{\pi}$. For example, the permutation $\pi=1472365$ in
$\mathfrak{S}_7$ has the standard cycle decomposition
$(1)(24)(375)(6)$,  so $\hat{\pi}=1243756$,  it has exactly two
top-descents $4$ and $7$, and $q_{\pi}=7$.

We define a map $\Phi:\mathfrak{S}_n\mapsto\mathfrak{S}_n$ as
follows:

For any $\pi\in\mathfrak{S}_n$, if $q_{\pi}$ is the last element of
a cycle $C_i$ for some $i$, then let $\Phi(\pi)$ be the permutation
obtained from $\pi$ by erasing the right and left parentheses ``)("
after the number $q_{\pi}$ in the standard cycle decomposition of
$\pi$; otherwise, let $\Phi(\pi)$ be the permutation obtained from
$\pi$ by inserting a right parentheses ``)" and a left parentheses
``(" after the number $q_{\pi}$ in the standard cycle decomposition
of $\pi$. For example, if $\pi=(1)(24)(375)(6)$, then
$\hat{\pi}=1243756$ and $q_{\pi}=7$, and so
$\Phi(\pi)=(1)(24)(37)(5)(6)$. If $\sigma=(1)(24)(37)(5)(6)$, then
$\Phi(\sigma)=(1)(24)(375)(6)$. Clearly, we have
$$\hat{\pi}=\widehat{\Phi(\pi)},~q_{\pi}=q_{\Phi(\pi)},~\exc(\pi)=\exc(\Phi(\pi)),~\cdes(\pi)-\cdes(\Phi(\pi))=\pm1.$$

Denote by $\Omega_{n,1}$ the set of the permutations
$\pi\in\mathfrak{S}_{n,1}$ such that $\hat{\pi}=123,\ldots n$. It is
easy to verify that the weight of $\Omega_{n,1}$ is $(1+x)^{n-2}$.
For any  permutation $\pi\in
\mathfrak{S}_{n,1}\setminus{\Omega}_{n,1}$, we have
$\Phi(\pi)\in\mathfrak{S}_{n,1}\setminus{\Omega}_{n,1}$. So for any
$\pi\in\mathfrak{S}_{n,1}$, let
$$\psi_{n,1}(\pi)=\left\{\begin{array}{lll}\Phi(\pi)&\text{if}&\pi\in\mathfrak{S}_{n,1}\setminus{\Omega}_{n,1},\\
\pi&\text{if}&\pi\in\Omega_{n,1}.\end{array}\right.$$

Hence,
\begin{eqnarray*}
\sum\limits_{\pi\in\mathfrak{S}_{n,1}}x^{\exc(\pi)}(-1)^{\cdes(\pi)}&=&\sum\limits_{\pi\in \Omega_{n,1}}x^{\exc(\pi)}(-1)^{\cdes(\pi)}+\sum\limits_{\pi\in \mathfrak{S}_{n,1}\setminus{\Omega}_{n,1}}x^{\exc(\pi)}(-1)^{\cdes(\pi)}\\
&=&\sum\limits_{\pi\in
\Omega_{n,1}}x^{\exc(\pi)}(-1)^{\cdes(\pi)}=\sum\limits_{\pi\in
\Omega_{n,1}}x^{\exc(\pi)}=(1+x)^{n-2}.
\end{eqnarray*}
For example, we list all $\pi\in\mathfrak{S}_{4,1}$ and
$\psi_{4,1}(\pi)$ in Table 1. $$\begin{array}{|l|l|l|l|l|}
\hline \pi\in\mathfrak{S}_{4,1}&x^{\exc(\pi)}(-1)^{\cdes(\pi)}&\hat{\pi}&q_\pi&\psi_{4,1}(\pi)\\
\hline (1)(2)(3)(4)&1&1234&&(1)(2)(3)(4)\\
\hline (1)(23)(4)&x&1234&&(1)(23)(4)\\
\hline (1)(2)(34)&x&1234&&(1)(2)(34)\\
\hline (1)(234)&x^2&1234&&(1)(234)\\
\hline (1)(24)(3)&x&1243&4&(1)(243)\\
\hline (1)(243)&-x&1243&4&(1)(24)(3)\\
\hline\end{array}$$\begin{center}Table.1. Involution
$\psi_{4,1}$\end{center}

For $2\leq i\leq n$, denote by $\mathcal{A}_{n,i}$ the set of
permutations $\pi\in\mathfrak{S}_{n,i}$ such that the number $q_\pi$
isn't in the first cycle in the standard cycle decomposition of
$\pi$. Let $\mathcal{A}_{n,i}^*=\mathfrak{S}_{n,i}\setminus
\mathcal{A}_{n,i}$ for short. For any $\pi\in\mathcal{A}_{n,i}^*$,
let $\pi=C_1\ldots C_k$ be the standard cycle decomposition of
$\pi$, suppose the length of the cycle $C_1$ is $l+1$ and
$$\hat{\pi}=1,\ldots,i,a_1,\ldots,a_{n-l-1}.$$
Then we have
$$a_1<a_2<\ldots<a_{n-l-1}$$ since $q_\pi$ is an element
in the cycle $C_1$.

Now suppose that $C_1=(1,c_{11},\ldots, c_{1l})$. Let $Q$ be the set
of indices $j\in\{1,2,\ldots,l\}$ such that $c_{1j}$ is not the
largest number in the set
$\{1,2,\ldots,n\}\setminus\{c_{1,j+1},\ldots,c_{1l}\}$, i.e.,
$$Q=\{j\mid 1\leq j\leq l\text{ and } c_{1j}<\max
\{1,2,\ldots,n\}\setminus\{c_{1,j+1},\ldots,c_{1l}\} \}.$$ Let
$\Omega_{n,i}$ be the set of permutations
$\pi\in\mathcal{A}^*_{n,i}$ such that
 $Q=\emptyset$.
For any $i=2,\ldots,n-1$, we have $l\in Q$ since $i<n$, and so
$\Omega_{n,i}=\emptyset$. Moreover, $\pi\in\Omega_{n,n}$ if and only
if
$$\hat{\pi}=1,k,\ldots n,2,3,\ldots,k-1$$ for some $k>1$.

We define a map $\Psi$ from $\mathcal{A}^*_{n,i}\setminus
\Omega_{n,i}$ to itself as follows:

For any $\pi\in\mathcal{A}^*_{n,i}\setminus \Omega_{n,i}$, let
$m=m_{\pi}=\min Q$ since $Q\neq\emptyset$. Then we have
$$c_{11}<\cdots<c_{1,m-1}\text{ and }c_{1,m-1}>c_{1m}.$$
Furthermore, we distinguish the following two cases:

{\bf Case 1.} $2\leq m\leq l$.

Let $$\Psi(\pi)=(1,c_{1m},\ldots,c_{1l})\cdot C_2\ldots C_k\cdot
(c_{11},\ldots,c_{1,m-1}).$$ Then $\Psi(\pi)$ has at least two
cycles, $m_{\Phi(\pi)}=1$ since $c_{1m}<c_{1,m-1}$, and so
$\Psi(\pi)\in\mathcal{A}^*_{n,i}\setminus \Omega_{n,i}$. Moreover,
we have $\exc(\pi)=\exc(\Psi(\pi))$ and
$\cdes(\pi)=\cdes(\Psi(\pi))+1$

{\bf Case 2.} $m=1$ and there are at least two cycles in the
standard cycle decomposition of $\pi$.

Suppose that $C_1=(1,c_{11},\ldots, c_{1l})$ and
$C_k=(c_{k1},\ldots, c_{ks})$ are the first cycle and the last cycle
in the standard cycle decomposition of $\pi$ respectively, where $s$
is the length of the cycle $C_k$. Let $$\Psi(\pi)=(1,c_{k1},\ldots,
c_{ks},c_{11},\ldots,c_{1l})\cdot C_2\ldots C_{k-1}.$$ Then
$$m_{\Psi(\pi)}=s+1\geq 2,$$ and so
$\Psi(\pi)\in\mathcal{A}^*_{n,i}\setminus \Omega_{n,i}$. Moreover,
we have $\exc(\pi)=\exc(\Psi(\pi))$ and
$\cdes(\pi)=\cdes(\Psi(\pi))-1$

When $2\leq i\leq n-1$, for any $\pi\in\mathfrak{S}_{n,i}$, let
$$\psi_{n,i}(\pi)=\left\{\begin{array}{lll}\Phi(\pi)&\text{if}&\pi\in\mathcal{A}_{n,i}\\
\Psi(\pi)&\text{if}&\pi\in\mathfrak{S}_{n,i}\setminus\mathcal{A}_{n,i}\end{array}\right..$$
For example, we list all $\pi\in\mathfrak{S}_{4,2}$ and
$\psi_{4,2}(\pi)$ in Table 2, and $\pi\in\mathfrak{S}_{4,3}$ and
$\psi_{4,3}(\pi)$ in Table 3.$$\begin{array}{|l|l|l|l|l|l|}
\hline \pi\in\mathfrak{S}_{4,2}&x^{\exc(\pi)}(-1)^{\cdes(\pi)}&\hat{\pi}&q_\pi&m_\pi&\psi_{4,2}(\pi)\\
\hline (12)(3)(4)&x&1234&&1&(142)(3)\\
\hline (142)(3)&-x&1423&4&2&(12)(3)(4)\\
\hline (12)(34)&x^2&1234&&1&(1342)\\
\hline (1342)&-x^2&1342&4&3&(12)(34)\\
\hline (1432)&x&1432&3&2&(132)(4)\\
\hline (132)(4)&-x&1324&3&1&(1432)\\
\hline\end{array}$$\begin{center}Table.2. Involution $\psi_{4,2}$
$$\begin{array}{|l|l|l|l|l|l|}
\hline \pi\in\mathfrak{S}_{4,3}&x^{\exc(\pi)}(-1)^{\cdes(\pi)}&\hat{\pi}&q_\pi&m_\pi&\psi_{4,3}(\pi)\\
\hline (13)(2)(4)&x&1324&3&1&(143)(2)\\
\hline (143)(2)&-x&1432&3&2&(13)(2)(4)\\
\hline (13)(24)&x^2&1324&3&1&(1243)\\
\hline (1243)&-x^2&1243&4&3&(13)(24)\\
\hline (1423)&-x^2&1423&4&2&(123)(4)\\
\hline (123)(4)&x^2&1234&&1&(1423)\\
\hline\end{array}$$Table.3. Involution $\psi_{4,3}$\end{center}

Hence,
\begin{eqnarray*}\sum\limits_{\pi\in\mathfrak{S}_{n,i}}x^{\exc(\pi)}(-1)^{\cdes(\pi)}&=&\sum\limits_{\pi\in \mathcal{A}_{n,i}}x^{\exc(\pi)}(-1)^{\cdes(\pi)}+\sum\limits_{\pi\in \mathfrak{S}_{n,i}\setminus\mathcal{A}_{n,i}}x^{\exc(\pi)}(-1)^{\cdes(\pi)}
=0.\end{eqnarray*}

When $i=n$, it is not hard to see that the weight of $\Omega_{n,n}$
is $x(1+x)^{n-2}$. For any $\pi\in\mathfrak{S}_{n,n}$, let
$$\psi_{n,n}(\pi)=\left\{\begin{array}{lll}\Phi(\pi)&\text{if}&\pi\in\mathcal{A}_{n,n}\\
\Psi(\pi)&\text{if}&\pi\in\mathfrak{S}_{n,n}\setminus(\mathcal{A}_{n,n}\cup\Omega_{n,n})\\
\pi&\text{if}&\pi\in\Omega_{n,n}\end{array}\right..$$ For example,
we list all $\pi\in\mathfrak{S}_{4,4}$ and $\psi_{4,4}(\pi)$ in
Table 4.$$\begin{array}{|l|l|l|l|l|l|}
\hline \pi&x^{\exc(\pi)}(-1)^{\cdes(\pi)}&\hat{\pi}&q_\pi&m_\pi&\psi_{4,4}(\pi)\\
\hline (14)(2)(3)&x&1423&4&&(14)(2)(3)\\
\hline (14)(23)&x^2&1423&4&&(14)(23)\\
\hline (134)(2)&x^2&1342&4&&(134)(2)\\
\hline (1234)&x^3&1234&&&(1234)\\
\hline
\hline (124)(3)&x^2&1243&4&1&(1324)\\
\hline (1324)&-x^2&1324&3&2&(124)(3)\\
\hline\end{array}$$\begin{center}Table.4. Involution
$\psi_{4,4}$\end{center}

 Hence,
\begin{eqnarray*}&&\sum\limits_{\pi\in\mathfrak{S}_{n,n}}x^{\exc(\pi)}(-1)^{\cdes(\pi)}\\
&=&\sum\limits_{\pi\in \mathcal{A}_{n,n}}x^{\exc(\pi)}(-1)^{\cdes(\pi)}+\sum\limits_{\pi\in \mathfrak{S}_{n,n}\setminus(\mathcal{A}_{n,n}\cup\Omega_{n,n})}x^{\exc(\pi)}(-1)^{\cdes(\pi)}+\sum\limits_{\pi\in \Omega_{n,n}}x^{\exc(\pi)}(-1)^{\cdes(\pi)}\\
&=&\sum\limits_{\pi\in
\Omega_{n,n}}x^{\exc(\pi)}(-1)^{\cdes(\pi)}=x(1+x)^{n-2}.\end{eqnarray*}

\end{proof}

\section{Proof of the explicit formula~\ref{P_{n,i}(x,-1,0,t)} in Theorem~\ref{theorem-p}}\label{section04}
Let $P_{n}(x,y,0,1)=\sum\limits_{i=1}^nP_{n,i}(x,y,0,1)$. We first
give the recurrence for $P_{n}(x,y,0,1)$.

\begin{lemma}\label{lemma-P(n,i)(x,y,0,1)-recurrent}For any $n\geq 2$ and $2\leq i\leq n+1$,
we have
\begin{equation*}P_{n+1,i}(x,y,0,1)=xP_{n-1}(x,y,0,1)+x\sum\limits_{j=2}^{i-1}P_{n,j}(x,y,0,1)+y\sum\limits_{j=i}^{n}P_{n,j}(x,y,0,1).\end{equation*}
\end{lemma}

\begin{proof}
 For any $\pi=\pi(1)\pi(2)\ldots\pi(n+1)\in\mathscr{D}_{n+1,i}$, let
$\sigma=(1,c_1,c_2,\ldots,c_l)$ be the cycle in the standard cycle
decomposition of $\pi$ which contains the number $1$. So $\pi$ can
be split into the cycle $\sigma$ and a permutation $\tau$ on the set
$\{1,2,\ldots,n+1\}\setminus\{1,c_1,\ldots,c_l\}$, i.e.,
$\pi=\sigma\cdot\tau$. Clearly, $l\geq 1$, $i\geq 2$ and $c_l=i$
since $\pi\in\mathfrak{S}_{n+1,i}$. We distinguish the following two
cases:

{\bf Case 1.} $l=1$.

Deleting the cycle $(1,c_1)=(1,i)$ from the standard cycle
decomposition of $\pi$, we obtain the permutation
$$\tau=\pi(2)\ldots\pi(i-1)\pi(i+1)\ldots\pi(n+1)$$ which is defined on the
set $\{2,\ldots,i-1,i+1,\ldots,n+1\}$. Note that
$\red(\tau)\in\mathscr{D}_{n-1}$,
$$\exc(\pi)=\exc(\red(\tau))+1\text{ and }\cdes(\pi)=\cdes(\red(\tau)).$$
This provides the term $$xP_{n-1}(x,y,0,1).$$

{\bf Case 2.} $l\geq 2$.

Suppose that $c_{l-1}=j$ for some $2\leq j\leq n+1$. Deleting the
number $c_l=i$ from  the standard cycle decomposition of $\pi$, we
obtain a permutation
$$\tilde{\pi}=(1,c_1,\ldots,c_{l-1})\cdot\tau$$ which is defined on
the set $\{1,\ldots,i-1,i+1,\ldots,n+1\}$. Note that
$\red(\tilde{\pi})\in\mathscr{D}_{n}$. Moreover, if $c_{l-1}=j\leq
i-1$, then
$$\red(\tilde{\pi})\in\mathscr{D}_{n,j},
\exc(\pi)=\exc(\red(\tilde{\pi}))+1,\cdes(\pi)=\cdes(\red(\tilde{\pi})).$$
This provides the term $$x\sum\limits_{j=2}^{i-1}P_{n,j}(x,y,0,1).$$
If $c_{l-1}=j\geq i+1$, then
$$\red(\tilde{\pi})\in\mathscr{D}_{n,j-1},
\exc(\pi)=\exc(\red(\tilde{\pi})),\cdes(\pi)=\cdes(\red(\tilde{\pi}))+1.$$
This provides the term
$$y\sum\limits_{j=i}^{n}P_{n,j}(x,y,0,1).$$
Thus, for any $i\geq 2$ we have $$P_{n+1,i}(x,y,0,1)=
xP_{n-1}(x,y,0,1)+x\sum\limits_{j=2}^{i-1}P_{n,j}(x,y,0,1)+y\sum\limits_{j=i}^{n}P_{n,j}(x,y,0,1).$$\end{proof}

\noindent{\bf A proof of the identity~\ref{P_{n,i}(x,-1,0,t)} in
Theorem~\ref{theorem-p}:}
\begin{proof} Note that
$$\sum_{\pi\in
\mathscr{D}_{n,i}}x^{\exc(\pi)}(-1)^{\cdes(\pi)}t^{\pi^{-1}(1)}=t^{i}\sum_{\pi\in
\mathscr{D}_{n,i}}x^{\exc(\pi)}(-1)^{\cdes(\pi)}=t^iP_{n,i}(x,-1,0,1).$$
Therefore, it is sufficient to show that
\begin{equation}\label{equation-P(n,i)(x,-1,0,1)-formula-proof}P_{n,i}(x,-1,0,1)=(-1)^{n-i}x^{i-1}\end{equation} for any $n\geq
2$ and $2\leq i\leq n$.

\noindent{\bf (i) An induction proof of the explicit formula~\ref{equation-P(n,i)(x,-1,0,1)-formula-proof}:}\\
It is easy to check that
$$P_{2,2}(x,-1,0,1)=x.$$ Assume that the formula holds
for any $2\leq k\leq n$. By
Lemma~\ref{lemma-P(n,i)(x,y,0,1)-recurrent}, we have
\begin{eqnarray*}&&
P_{n+1,i}(x,-1,0,1)\\&=&xP_{n-1}(x,-1,0,1)+x\sum\limits_{j=2}^{i-1}P_{n,j}(x,-1,0,1)-\sum\limits_{j=i}^{n}P_{n,j}(x,-1,0,1)\\
&=&x\sum\limits_{j=2}^{n-1}P_{n-1,j}(x,-1,0,1)+x\sum\limits_{j=2}^{i-1}P_{n,j}(x,-1,0,1)-\sum\limits_{j=i}^{n}P_{n,j}(x,-1,0,1)\\
&=&\sum\limits_{j=2}^{n-1}(-1)^{n-1-j}x^j+\sum\limits_{j=2}^{i-1}(-1)^{n-j}x^j-\sum\limits_{j=i}^{n}(-1)^{n-j}x^{j-1}\\
&=&(-1)^{n+1-i}x^{i-1}
\end{eqnarray*}
for any $2\leq i\leq n$.

\noindent{\bf (ii) A bijective proof of the explicit formula~\ref{equation-P(n,i)(x,-1,0,1)-formula-proof}:}\\
Next we give a bijective proof of the explicit
formula~\ref{equation-P(n,i)(x,-1,0,1)-formula-proof} by
establishing an involution $\varphi_{n,i}$ on $\mathscr{D}_{n,i}$.
Fix $i\in\{2,\ldots,n\}$. By definition, the weight of each
$\pi\in\mathscr{D}_{n,i}$ is $(-1)^{\cdes(\pi)}x^{\exc(\pi)}$, hence
the weight of the cyclic permutation
$$\sigma^i=(1,2,\ldots,i-1,n,n-1,\ldots,i)\in\mathscr{D}_{n,i}$$ is $(-1)^{n-i}x^{i-1}$.

For any $\pi\in\mathscr{D}_{n,i}$, suppose that $\pi=C_1\ldots C_k$
is the standard cycle decomposition of $\pi$ and
$$C_k=(c_{k1},\ldots,c_{ks}).$$
We distinguish the following three cases:

 {\bf Case 1.} $k=1$ and
$C_k=(1,2,\ldots,i-1,n,n-1,\ldots,i)$.

Then let $\varphi_{n,i}(\pi)=\pi$.

{\bf Case 2.} $k\geq 2$ and $\red(C_k)=
(1,2,\ldots,r-1,s,s-1,\ldots,r)$ for some $r=2,3,\ldots,s$.

Suppose that $C_{k-1}=(c_{k-1,1},c_{k-1,2},\ldots,c_{k-1,t})$ and
$c_{k,j}$ is the largest number in the set
$\{c_{k,1},c_{k,2},\ldots,c_{k,l_k}\}$ for some
$j\in\{1,2,\ldots,l_k\}$. If $c_{k-1,2}<c_{k,j-1}$, then let
$$\varphi_{n,i}(\pi)=C_1\ldots
C_{k-2}\cdot(c_{k-1,1},c_{k1},c_{k2},\ldots,c_{ks},c_{k-1,2},\ldots,c_{k-1,t}),$$
and so we have $$\exc(\pi)=\exc(\varphi_{n,i}(\pi))\text{ and
}\cdes(\pi)=\cdes(\varphi_{n,i}(\pi))+1;$$ if $c_{k-1,2}>c_{k,j-1}$
then let
$$\varphi_{n,i}(\pi)=C_1\ldots
C_{k-2}\cdot(c_{k-1,1},c_{k1},\ldots,c_{k,j-2},c_{kj}\ldots,c_{ks},c_{k,j-1},c_{k-1,2},\ldots,c_{k-1,t}),$$
and so we have $$\exc(\pi)=\exc(\varphi_{n,i}(\pi))\text{ and
}\cdes(\pi)=\cdes(\varphi_{n,i}(\pi))+1.$$

{\bf Case 3.} $\red(C_k)\neq (1,2,\ldots,r-1,s,s-1,\ldots,r)$ for
any $r=2,3,\ldots,s$.

There exists a unique index $\tilde{s}$ such that
$$\red(c_{k1},c_{k2},\ldots,c_{k\tilde{s}})=1,2,\ldots,{r}-1,\tilde{s},\tilde{s}-1,\ldots,{r}$$
for some ${r}=2,3,\ldots,\tilde{s}$ and
$$\red(c_{k1},c_{k2},\ldots,c_{k,\tilde{s}+1})\neq 1,2,\ldots,\tilde{r}-1,\tilde{s}+1,\tilde{s},\ldots,\tilde{r}$$
for any $\tilde{r}=2,3,\ldots,\tilde{s}+1$. It is easy to check
$3\leq \tilde{s}\leq s-1$. Moreover, suppose that $c_{kj}$ is the
largest number in the set $\{c_{k1},c_{k2},\ldots,c_{k\tilde{s}}\}$.
Then we have $$c_{k,\tilde{s}+1}<c_{k,j-1}\text{ or
}c_{k,\tilde{s}+1}>c_{k\tilde{s}}.$$ If
$c_{k,\tilde{s}+1}<c_{k,j-1}$ then
$$\varphi_{n,i}(\pi)=C_1\ldots
C_{k-2}\cdot(c_{k1},c_{k,\tilde{s}+1},\ldots,c_{ks})\cdot(c_{k2},\ldots,c_{k\tilde{s}}),$$
 we have $$\exc(\pi)=\exc(\varphi_{n,i}(\pi))\text{ and }\cdes(\pi)=\cdes(\varphi_{n,i}(\pi))-1;$$ if $c_{k,\tilde{s}+1}>c_{k\tilde{s}}$
 then  let
$$\varphi_{n,i}(\pi)=C_1\ldots
C_{k-2}\cdot(c_{k1},c_{k,\tilde{s}+1},\ldots,c_{ks})\cdot(c_{k2},\ldots,c_{k,j-1},c_{k,\tilde{s}},c_{kj},\ldots,c_{k,\tilde{s}-1}),$$
we have $$\exc(\pi)=\exc(\varphi_{n,i}(\pi))\text{ and
}\cdes(\pi)=\cdes(\varphi_{n,i}(\pi))-1.$$ For the case with $n=4$,
we list all $\pi$ and $\varphi_{n,i}(\pi)$ in Table. 5.
$$\begin{array}{||l|l||l|l||l|l||}\hline
\pi\in\mathscr{D}_{4,2}&\varphi_{4,2}(\pi)&\pi\in\mathscr{D}_{4,3}&\varphi_{4,4}(\pi)&\pi\in\mathscr{D}_{4,4}&\varphi_{4,4}(\pi)\\
\hline (12)(34)&(1342)&(13)(24)& (1423)&(14)(23)&(1324)\\
\hline (1342)&(12)(34)&(1423)&(13)(24)&(1324)&(14)(23) \\
\hline (1432)&(1432)&(1423)&(1423)&(1234)&(1234) \\
\hline\end{array}$$\begin{center}Table. 5. Involutions
$\varphi_{n,i}(\pi)$ for $n=4$\end{center}

Hence,
\begin{eqnarray*}
\sum\limits_{\pi\in\mathscr{D}_{n,i}}x^{\exc(\pi)}(-1)^{\cdes(\pi)}&=&(-1)^{n-i}x^{i-1}+\sum\limits_{\pi\in\mathscr{D}_{n,i}\setminus
\{\sigma^i\}}x^{\exc(\pi)}(-1)^{\cdes(\pi)}
=(-1)^{n-i}x^{i-1}.\end{eqnarray*}\end{proof}

\section{Proof of the recurrence relation~\ref{P_{n,i}(1,y,1,1)-recurrent}}\label{section05}
Suppose that $y$ is a positive integer. Let $\mathfrak{S}_n(y)$
denote the set of pairs $[\pi,\phi]$ such that
$\pi\in\mathfrak{S}_n$ and $\phi$ is a map from the set $CDES(\pi)$
to the set $\{0,1,\cdots,y-1\}$. So,
$b_{n}(y,1)=|\mathfrak{S}_{n}(y)|$.

For any $[\pi,\phi]\in\mathfrak{S}_{n+1}(y)$, we distinguish the
following two cases:

 {\bf Case 1.} $\pi(1)=1$.

 Let $\tau=\pi(2)\ldots\pi(n+1)$. Then $\tau$ is a permutation defined on the set
 $\{2,3,\ldots,n+1\}$
 and
$$\red(\tau)\in\mathfrak{S}_{n}.$$ Define a map
$\phi':[n]\mapsto\{+1,-1\}$ by letting $\phi'(i)=\phi(\red^{-1}(i))$
for $i=1,2,\ldots,n$. Then
$$[\red(\tau),\phi']\in\mathfrak{S}_n(y),$$ and so this provides the
term $b_n(y,1)$.

 {\bf Case 2.} $\pi(1)\neq 1$.

Let $\sigma=(1,c_1,c_2,\ldots,c_l)$ be the cycle in the standard
cycle decomposition of $\pi$ which contains the number $1$. So,
$\pi$ is split into the cycle $\sigma$ and a permutation $\tau$ on
the set $\{1,2,\ldots,n+1\}\setminus\{1,c_1,\ldots,c_l\}$, i.e.,
$\pi=\sigma\cdot\tau$. Clearly, $l\geq 1$ since $\pi(1)\neq 1$.

Note that there is a unique index $k\geq 1$ which satisfies
$c_{k-1}<c_k$ and $c_{k}>c_{k+1}>\cdots
>c_{l}$. For the sequence $c_k\ldots c_l$, if
$\phi(c_i)=0$ for some $k\leq i\leq l-1$ then let $k'$ be the
largest index in $\{k,k+1,\ldots,l-1\}$ such that $\phi(c_{k'})=0$;
otherwise, $k'=k-1$. Let
$$\sigma'=(1,c_1,\ldots,c_{k'})\text{ and }\pi'=\sigma'\cdot \tau.$$
Then $\pi'$ is a permutation defined on the set $[n+1]\setminus B$,
where $$ B=\{c_{k'+1},\ldots,c_l\},$$ and
$$\red(\pi')\in\mathfrak{S}_{n+1-|B|}.$$
Define a map $\phi':[n+1-|B|]\mapsto\{+1,-1\}$ by letting
 $$\phi'(i)=\phi(\red^{-1}(i))$$ for any $1\leq i\leq n+1-|B|$. Then
   $$[\red(\pi'),\phi']\in\mathfrak{S}_{n+1-|B|}(y).$$

Note that $1\leq |B|\leq n$ and $B\setminus\{c_l\}\subseteq
CDES_{n+1}(\pi)$. For any $k\leq i\leq l-1$, let
$\theta(c_i)=\phi(c_i)$. Then $\theta$ is a map from the set
$\{c_{k'+1},\ldots,c_{l-1}\}$ to $\{1,2,\ldots,y-1\}$. So there are
${\binom{n}{|B|}}$ ways to form the set $B$ and $(y-1)^{|B|-1}$ ways
to form the map $\theta$. This provides the term
$$\sum\limits_{i=1}^nb_{n+1-i}(y,1){\binom{n}{i}}(y-1)^{i-1}$$

Hence we derive the recurrence relation
\begin{eqnarray*}b_{n+1}(y,1)&=&b_n(y,1)+\sum\limits_{i=1}^nb_{n+1-i}(y,1){\binom{n}{i}}(y-1)^{i-1}\\
&=&b_n(y,1)+\sum\limits_{i=1}^nb_{i}(y,1){\binom{n}{i-1}}(y-1)^{n-i}.\end{eqnarray*}

\section{Proof of the recurrence relation~\ref{P_{n,i}(1,y,0,1)-recurrent}}\label{section06}
Clearly, we have $b_0(y,0)=1$ and $b_1(y,0)=0$. Suppose that $y$ is
a positive integer. Let $\mathscr{D}_n(y)$ denote the set of pairs
$[\pi,\phi]$ such that $\pi\in\mathscr{D}_n$ and $\phi$ is a map
from the set $CDES(\pi)$ to the set $\{0,1,\cdots,y-1\}$. Hence
$b_{n}(y,0)=|\mathscr{D}_{n}(y)|$.

For any $[\pi,\phi]\in\mathscr{D}_{n+1}(y)$, let
$\sigma=(1,c_1,c_2,\ldots,c_l)$ be the cycle in the standard cycle
decomposition of $\pi$ which contains the number $1$. So, $\pi$ is
split into the cycle $\sigma$ and a permutation $\tau$ on the set
$[n+1]\setminus\{1,c_1,\ldots,c_l\}$, i.e., $\pi=\sigma\cdot\tau$.
Clearly,  $l\geq 1$ since $\pi(1)\neq 1$.

Note that there is a unique index $k\geq 1$ which satisfies
$c_{k-1}<c_k$ and $c_{k}>c_{k+1}>\cdots
>c_{l}$. For the sequence $c_k\ldots c_l$, if
$\phi(c_i)=0$ for some $k\leq i\leq l-1$ then let $k'$ be the
largest index in $\{k, k+1,\ldots,l-1\}$ such that $\phi(c_{k'})=0$;
otherwise, $k'=k-1$.

We distinguish the following two cases:

{\bf Case 1.} $k'=0$

Let $$B=\{c_{1},\ldots,c_l\}.$$ Note that $\tau$ is a permutation
defined on the set $[n+1]\setminus \{1,c_1,\ldots,c_l\}$ and
$$\red(\tau)\in \mathfrak{S}_{n-|B|}.$$ Define a map
$\phi':[n-|B|]\mapsto\{+1,-1\}$ by letting
$$\phi'(i)=\phi(red^{-1}(i))$$ for any $1\leq i\leq n-|B|$.
Then $$[\red(\tau),\phi']\in\mathfrak{S}_{n-|B|}(y)$$ and there are
$b_{n-|B|}(y,0)$ ways to form the pairs $[\red(\tau),\phi']$.

Note that $1\leq |B|\leq n$ and $B\setminus\{c_l\}\subseteq
CDES_{n+1}(\pi)$. For any $k\leq i\leq l-1$, let
$\theta(c_i)=\phi(c_i)$. Then $\theta$ is a map from the set
$\{c_{k},\ldots,c_{l-1}\}$ to $\{1,2,\ldots,y-1\}$. So there are
${\binom{n}{|B|}}$ ways to form the set $B$ and $(y-1)^{|B|-1}$ ways
to form the mapping $\theta$.

This provides the term
\begin{eqnarray*}\sum\limits_{i=1}^nb_{n-i}(y,0){\binom{n}{i}}(y-1)^{i-1}.\end{eqnarray*}

{\bf Case 2.} $k'\geq 1$

Let
$$\sigma'=(1,c_1,\ldots,c_{k'})\text{ and }\pi'=\sigma'\cdot
\tau.$$
    Then $\pi'$ is a permutation defined on the set $[n+1]\setminus B$, where $$B=\{c_{k'+1},\ldots,c_l\},$$  and
$$\red(\pi')\in\mathfrak{S}_{n+1-|B|}.$$ Define a map
$\phi':[n+1-|B|]\mapsto\{+1,-1\}$ by letting
 $$\phi'(i)=\phi(\red^{-1}(i))$$ for any $1\leq i\leq n+1-|B|$.
Then $$[\red(\pi'),\phi']\in\mathscr{D}_{n+1-|B|}(y)$$ and there are
$b_{n+1-|B|}(y,0)$ ways to form the pairs $(\red(\pi'),\phi')$.

Note that $1\leq |B|\leq n-1$ and $B\setminus\{c_l\}\subseteq
CDES(\pi)$. For any $k\leq i\leq l-1$, let $\theta(c_i)=\phi(c_i)$.
Then $\theta$ is a map from the set $\{c_{k'+1},\ldots,c_{l-1}\}$ to
$\{1,2,\ldots,y-1\}$. So there are ${\binom{n}{|B|}}$ ways to form
the set $B$ and $(y-1)^{|B|-1}$ ways to form the map $\theta$.

This provides the term
\begin{eqnarray*}\sum\limits_{i=1}^{n-1}b_{n+1-i}(y,0){\binom{n}{i}}(y-1)^{i-1}.\end{eqnarray*}

Hence, we have
\begin{eqnarray*}b_{n+1}(y,0)&=&\sum\limits_{i=1}^{n}{\binom{n}{i}}b_{n-i}(y,0)(y-1)^{i-1}+\sum\limits_{i=1}^{n-1}{\binom{n}{i}}b_{n+1-i}(y,0)(y-1)^{i-1}.\end{eqnarray*}

\section{Proof of Theorem \ref{theorem-bijection-p-m}}\label{section07}

\begin{lemma}\label{lemma-cyc-des-callan-match} There is a bijection $\Theta_n$ from the set of cyclic
negative cycle descent  permutations  of $[n]$ to the set of
connected Callan perfect matchings of $\mathbb{P}_n$.
\end{lemma}
\begin{proof} Let $(\pi,\phi)$ be a cyclic
negative cycle descent  permutations of $[n]$. Then there is exactly
one cycle $C$ in the standard cycle decomposition of $\pi$. Suppose
$$C=(c_1,c_2,\ldots,c_n)$$ where $c_1=1$. Erase the parentheses,
draw a bar after each element  $c_i$ which has sign $+1$, and add a
bar before $c_1$. Regard the numbers between two consecutive bars as
``blocks". So, we decompose $(\pi,\phi)$ into a sequence of blocks
$$B_1,B_2,\ldots,B_k.$$ Suppose that the $i$-th block $B_i$ contains $t_i$ number
$b_{i1},\ldots,b_{it_i}$ with $b_{i1}>\ldots>b_{it_i}$. We construct
a perfect matchings $M$ as follows:
\begin{itemize}
\item Step 1. For every block $B_i$, we connect the vertex $(b_{i,j},0)$ to the vertex $(b_{i,j+1},1)$ as a downline of $M$  for any $1\leq j\leq t_i-1$.
\item Step 2. For any odd integer $i\in\{1,2,\ldots,k-1\}$, we connect the vertex $(b_{i,t_i},0)$ to the vertex $(b_{i+1,t_{i+1}},0)$ as an arc of $M$.
For any even integer $i\in\{1,2,\ldots,k-1\}$, we connect the vertex
$(b_{i,1},1)$ to the vertex $(b_{i+1,1},1)$ as an arc of $M$.
\item Step 3. If $k$ is odd, we connect the vertex $(b_{1,1},1)=(1,1)$ to the vertex
$(b_{k,t_k},0)$ as a downline of $M$; otherwise, connect the vertex
$(b_{1,1},1)=(1,1)$ to the vertex $(b_{k1},1)$ as an arc of $M$.
\end{itemize}
It is easy to check that $M$ is connected and has no uplines. So,
$M$ is a connected Callan perfect matching. Define $\Theta_n$ as a
map from the set of cyclic negative cycle descent permutations of
$[n]$ to the set of connected Callan perfect matchings of
$\mathbb{P}_n$ by letting $\Theta_n(\pi,\phi)=M$. Let $(\pi,\phi)$
and $(\pi',\phi')$ be two different cyclic negative cycle descent
permutations  of $[n]$. Then the sequence of  blocks of $(\pi,\phi)$
and $(\pi',\phi')$ are different. This implies
$\Theta_n(\pi,\phi)\neq \Theta_n(\pi',\phi')$, and so the map
$\Theta_n$ is a bijection.

Conversely, let $M$ be a connected Callan perfect matching of
$\mathbb{P}_n$. Delete the edge incident with the vertex $(1,1)$
from $M$, identify two vertices $(i,0)$ and $(i,1)$ in $M$ as a new
vertex $i$ for each $i=1,2,\ldots,n$, denote by $\mathcal {G}^*(M)$
the graph obtained from $M$. Then the graph $\mathcal {G}^*(M)$ is a
path on the vertex set $[n]$ and can be written as
$$a_1a_2\ldots a_n$$ where  $a_1=1$ and the set
$\{a_1a_{2},a_3a_{4},\ldots,a_{n-1}a_{n}\}$ is the edge set of
$\mathcal {G}^*(M)$. Draw a bar after each number $a_i$ which
satisfies either (1) $i=n$ or (2) there is an arc of $M$ in
$$\{\{(a_i,0),(a_{i+1},0)\}, \{(a_i,1),(a_{i+1},1)\}\},$$ and add a
bar before $a_1$. Regard the numbers between two consecutive bars as
``blocks". So, we obtain a sequence of blocks
$$B_1',B_2',\ldots,B_k'.$$ We construct  a cyclic negative
cycle descent  permutations $(\pi,\phi)$ of $[n]$ as follows:
\begin{itemize}
\item Step $1'$. For each block $B_i'$, we write the numbers in $B_i'$ in decreasing order, denote by $\tau_i$ the obtained
sequence, and let $$\pi=(\tau_1,\tau_2,\ldots,\tau_k).$$
\item Step $2'$. For any number $j\in[n]$, suppose $j$ is in a
block $B_i'$ for some $1\leq i\leq k$. If $j$ is the smallest number
in $B_i'$, then let the sign of $j$ be $+1$; otherwise, let the sign
of $j$ be $-1$. In fact, this defines a map $\phi$ from $[n]$ to
$\{+1,-1\}$.
\end{itemize}
Then $(\pi,\phi)$ is a cyclic negative cycle descent permutations of
$[n]$.
\end{proof}

\begin{ex} Let us consider a cyclic negative cycle descent permutation $$(1^+6^-4^-3^+2^+8^-7^-5^+)$$ of
$[8]$. We erase the parentheses, draw a bar after each element
 which has sign $+1$, and add a bar before $1$. Thus we
obtain
$$|1|643|2|875|$$ and the sequence of blocks $$B_1=1,B_2=643,B_3=2,B_4=875.$$
By Steps 1,2,and 3 in the proof of Lemma
\ref{lemma-cyc-des-callan-match}, we construct the following dot
diagram.
\begin{center}\includegraphics[width=6.5cm,height=2cm]{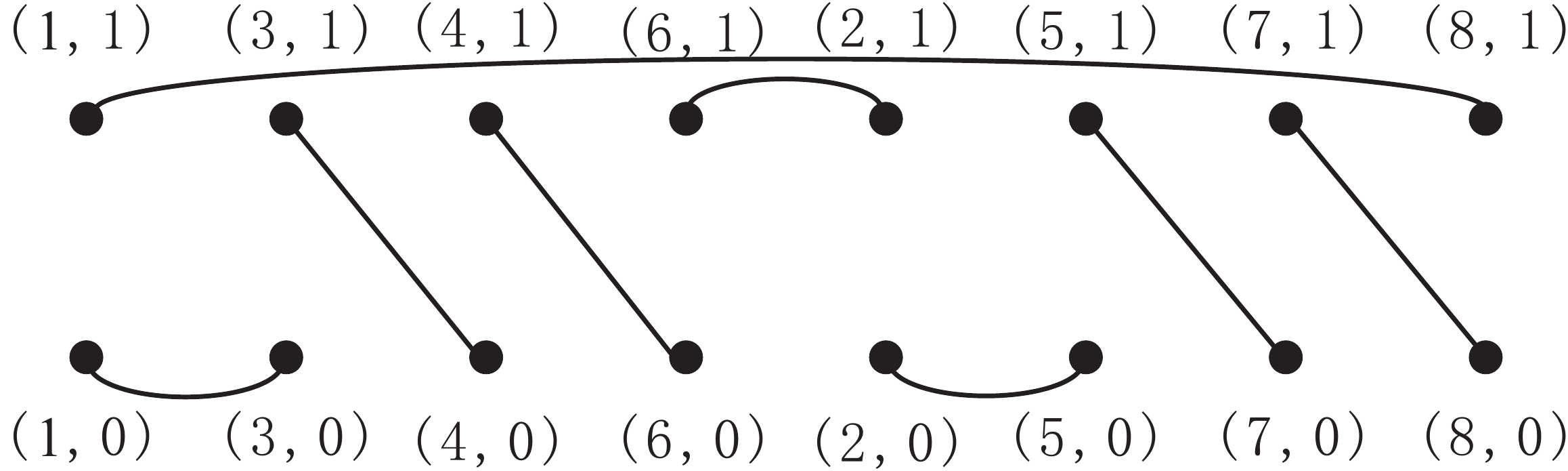}\\
Fig.4. A dot diagram constructed by Step 1,2, and 3 in the proof of
Lemma \ref{lemma-cyc-des-callan-match}  \end{center} We obtain a
connected Callan perfect matching $M$  as follows:
\begin{center}\includegraphics[width=6.5cm,height=2cm]{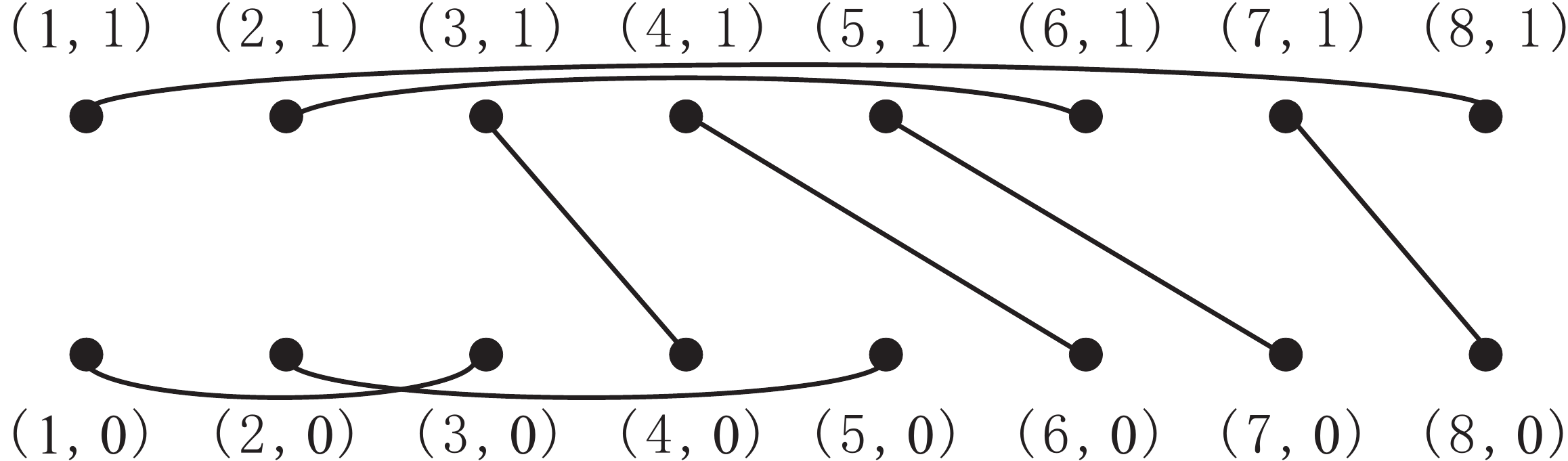}\\
Fig.5. A connected Callan perfect matching $M$ corresponding with
$(1^+6^-4^-3^+2^+8^-7^-5^+)$  \end{center}

Conversely, let us consider the connected perfect matching $M$ in
Fig.5. After deleting the edge $\{(1,1),(8,1)\}$, we can obtain the
graph
$$\mathcal{G}^*(M)=13462578$$ which has the edge set
$\{13,34,46,62,25,57,78\}$. Note that there is $3$ arcs
$$\{(1,0),(3,0)\},\{(6,1),(2,1)\},\{(2,0),(5,0)\}$$ in $M$. So, we
draw bars after the numbers $1,6,2,8$, and add a bar before $1$.
Thus we obtain
$$|1|346|2|578|$$ and the sequence of blocks
$$B_1'=1,B_2'=346,B_3'=2,B_4'=578.$$
By Steps $1'$ and $2'$ in the proof of Lemma
\ref{lemma-cyc-des-callan-match}, we construct a cyclic negative
cycle descent permutation $$(1^+6^-4^-3^+2^+8^-7^-5^+)$$ of $[8]$.
\end{ex}

\noindent{\bf A bijection proof of Theorem
\ref{theorem-bijection-p-m}}
\begin{proof} Let $(\pi,\phi)$ be a
negative cycle descent  permutations  of $[n]$. Suppose that
$\pi=C_1\ldots C_k$ is the  standard cycle decomposition of $\pi$
and $$C_i=(c_{i1},\ldots,c_{i,l_i})$$ for each $i=1,2,\ldots,k$.
Then
$$\red(C_i)\in\mathfrak{S}_{l_i},$$
Define a map
  $\phi^i:[l_i]\mapsto\{+1,-1\}$ by letting $$\phi^i(j)=\phi(\red^{-1}(j)),$$ i.e., the sign
of $\red(c_{ij})$ is the same as that of $c_{ij}$. Then
$$(\red(C_i),\phi^i)$$ is a cyclic negative cycle descent  permutations
of $[l_i]$. By Lemma \ref{lemma-cyc-des-callan-match}, we have
$\Theta_{l_i}(\red(C_i),\phi^i)$ is a connected Callan perfect
matching. For any $1\leq j\leq l_i$, we replace the labels $(j,0)$
and $(j,1)$ of vertices in $\Theta_{l_i}(\red(C_i),\phi^i)$ with
$(\red^{-1}(j),0)$ and $(\red^{-1}(j),1)$ respectively and denote by
$M^i$ the perfect matching obtained from
$\Theta_{l_i}(\red(C_i),\phi^i)$. At last, let $$M=M^1\cup M^2\cup
\ldots \cup M^k$$ where the notation $M\cup M'$ denotes the union of
two perfect matchings $M$ and $M'$ such that the vertex set of
$M\cup M'$ is $V(M)\cup V(M')$ and the edge set of $M\cup M'$ is
$E(M)\cup E(M')$. So $M$ is a Callan perfect matchings of
$\mathbb{P}_n$. Define $\Gamma_n$ as a map from the set of
 negative cycle descent permutations of $[n]$ to the set of
 Callan perfect matchings of $\mathbb{P}_n$ by letting
$\Gamma_n(\pi,\phi)=M$. Note that $\Gamma_n$ is injective, and so it
is a bijection.

By the definition of $\Gamma_n$, it is easy to see that
$$\com(\Gamma_n(\pi,\phi))=\cyc(\pi)\text{ and }
\ver(\Gamma_n(\pi,\phi))=\fix(\pi).$$ If the vertices $(1,1)$ and
its
  partner are in the same row, then
  $\down(\Gamma_n(\pi,\phi))=\nega(\pi,\phi)$; otherwise,
  $\down(\Gamma_n(\pi,\phi))=\nega(\pi,\phi)+1$.
\end{proof}

\begin{ex} Let us consider a negative cycle descent  permutation $$(1^+6^-3^+4^+)(2^+8^-7^+)(5^+)$$ of
$[8]$. We draw the perfect matchings $M^1$, $M^2$ and $M^3$
corresponding with the cycles $C_1$, $C_2$ and $C_3$ respectively as
follows:
$$\begin{array}{|c|c|c|c|}\hline Cycles&C_1&C_2&C_3\\
\hline &(1^+6^-3^+4^+)&(2^+8^-7^+)&(5^+)\\
\hline \text{Perfect
~~matchings}&M^1&M^2&M^3\\
\hline
&\includegraphics[width=3cm,height=1.8cm]{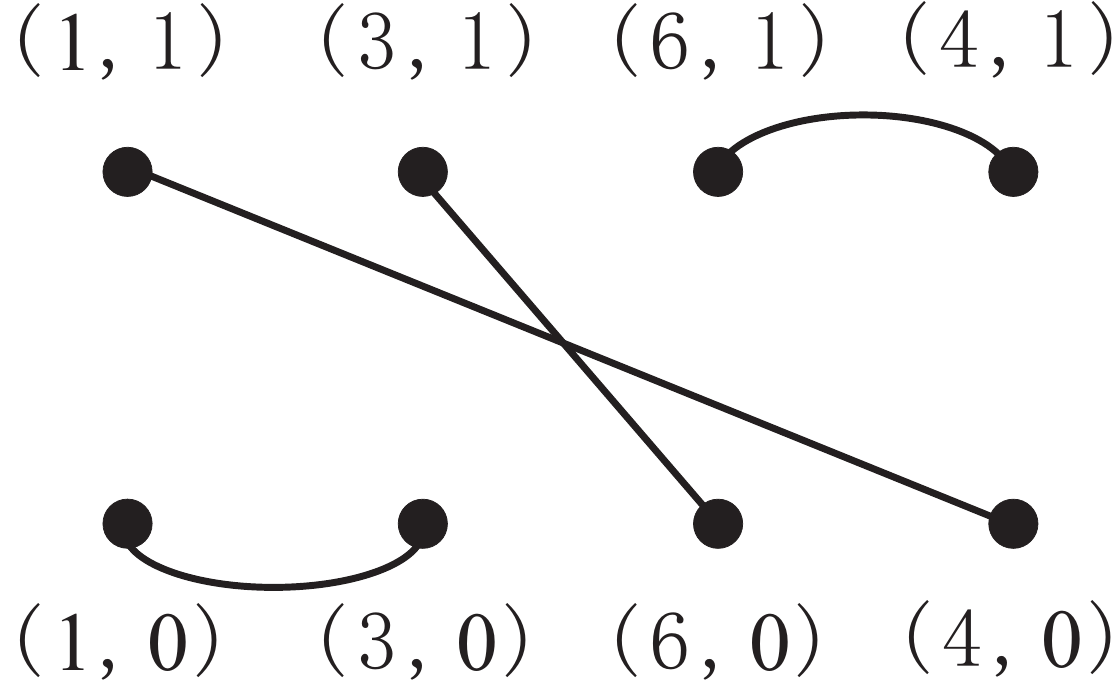}&\includegraphics[width=2.5cm,height=1.8cm]{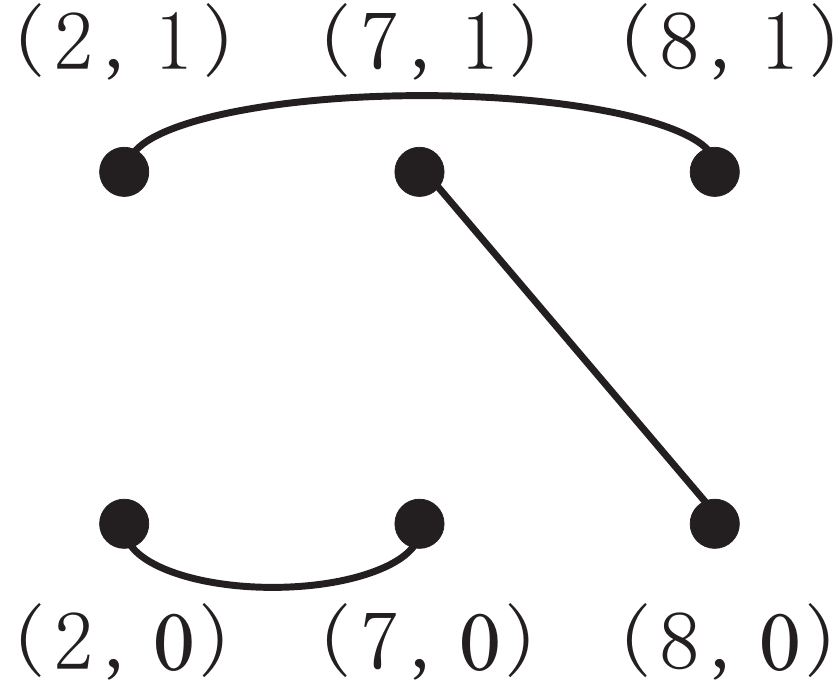}&\includegraphics[width=0.6cm,height=1.8cm]{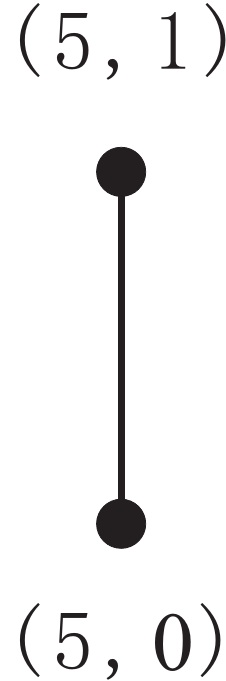}\\\hline\end{array}$$
Finally, we obtain a Callan perfect matching $M=M^1\cup M^2\cup M^3$
which is exactly that in Example \ref{example-callan-matching}.
\end{ex}

\end{document}